\documentclass[12pt]{article}
\usepackage{amsmath,amsfonts,amssymb,amsthm,amscd}
\usepackage{xcolor}
\title{Automorphism groups of prime models, and  invariant measures}
\date{\today}
\author{Anand Pillay \thanks{Partially supported by NSF grants DMS-1760212, and DMS-2054271}\\{University of Notre Dame}}

\newtheorem{Theorem}{Theorem}[section]
\newtheorem{Proposition}[Theorem]{Proposition}
\newtheorem{Definition}[Theorem]{Definition}
\newtheorem{Remark}[Theorem]{Remark}
\newtheorem{Lemma}[Theorem]{Lemma}

\newtheorem{Fact}[Theorem]{Fact}

\begin{document}
\maketitle

\begin{abstract}
We adapt the notion from \cite{KPR} and \cite{HKP} of a ( relatively) definable subset of $Aut(M)$ when $M$ is a  saturated structure, to the case $Aut(M/A)$  when $M$ is atomic and strongly $\omega$-homogeneous (over a set $A$). We discuss the existence and uniqueness of invariant measures on the Boolean algebra of definable subsets of $Aut(M/A)$.  For example when $T$ is stable, we have existence and uniqueness. 

We also discuss the compatibility of our definability notions with definable Galois cohomology from \cite{Pillay-Galois} and differential Galois theory.
\end{abstract}

\section{Introduction and preliminaries}

Typical examples of $M$ being atomic and strongly $\omega$-homogeneous over $A$, are when (a)  $A = \emptyset$ and $M$ the unique countable model of an $\aleph_{0}$-categorical (complete) countable theory, (b) $Th(M)$ is $\omega$-stable (or totally transcendental) and $M$ is prime over $A$, (c) $M = acl(A)$.  

The notion of a (relatively) definable subset of an automorphism group $Aut(M)$ was introduced in \cite{KPR} and later in \cite{HKP} where invariant measures were also considered, but all in  the context of $M$ being sufficiently saturated and homogeneous.  The reason for looking rather at prime models is that their   automorphism groups may support  ``Galois theories" of one kind or another. For example when  $M = acl(A)$, then $Aut(M/A)$ has the structure of a profinite group, the ``definable" subsets of $Aut(M/A)$ are precisely the clopens, there is a unique invariant measure on $Aut(M/A)$  (Haar measure) as well as a Galois correspondence (working in $M^{eq}$). 

In Section 2 we prove the main results on existence and uniqueness of invariant measures. Under the additional assumptions that $T$ is totally transcendental and $M$ is prime over $A$, we extend the results to ``$Aut(B/A)$" where $B\supseteq A$ is $Aut(M/A)$-invariant (also called ``normal").
In Section 3 we make several related observations. First that definable {\em subgroups} of $Aut(M/A)$ are precisely fixators of imaginaries. Secondly  that the notion of a definable subset of $Aut(M/A)$ is compatible with what, in \cite{Pillay-Galois}, we called definable cocycles.  Thirdly we specialize to differential Galois theory, remarking that when $Aut(L/K)$ has the structure of a (pro-)algebraic group $G$ (in the constants of $K$) then the two notions of definability coincide.

Our model theory notation is standard. $T$ denotes a complete theory and we work inside a very saturated model ${\bar M}$ of $T$. We will typically be working with $1$-sorted theories and structures, although imaginaries  also enter the picture.  As a rule $A, B$,.. denote small subsets of ${\bar M}$ and $M, N$,... small elementary submodels of ${\bar M}$. This paper will be discussing stable theories as well as $\omega$-stable (or totally transcendental) theories. For stable theories the reference is Chapter 1 of \cite{Pillay-GST}, and for $\omega$-stable theories,  Chapters 5 and 8 of \cite{Pillay-Introductionstability},  and  Chapter 18 of \cite{Poizat-modeltheory}. 

\begin{Definition} Let $M$ be a model and $A\subseteq M$. 
By a {\em definable} subset of $Aut(M/A)$ we mean $\{\sigma\in Aut(M/A): M\models \phi(\sigma({\bar b}), {\bar b})\}$ for some $L_{A}$-formula $\phi({\bar x}, {\bar y})$ with ${\bar x}$, ${\bar y}$ finite tuples of variables of the same length and ${\bar b}$ a tuple from $M$ of the same length. 
\end{Definition}  

This has nothing to do with definable subsets of the structure $(Aut(M/A), \cdot)$

%\begin{Remark} We should pay attention to the notation.  Suppose for example $M$ to be countable, with enumeration $(m_{1}, m_{2},.......)$.  So ${\bar x} = (x_{1},x_{2},...)$ and ${\bar y} = (y_{1},y_{2},...)$.
%Let $\phi({\bar x}, {\bar y})$ be an $L_{A}$-formula, then only finitely many variables actually appear in $\phi$, so we may write $\phi$ as $\phi(x_{1},..,x_{n}, y_{1},...,y_{n})$ for some $n$. The definable subset of $Aut(M/A)$ given by $\phi$ is then $\{\sigma\in Aut(M/A): M\models \phi(\sigma(m_{1}),..,\sigma(m_{n}),m_{1},...,m_{n})\}$. 
%\end{Remark}

Recall that the usual (``pointwise") topology on automorphism groups is as follows: a basic open set in $Aut(M/A)$ is $\{\sigma: \sigma({\bar a}) = {\bar b}\}$ where ${\bar a}$, ${\bar b}$ are finite tuples from $M$ of the same length. Note that such a basic open is definable in the sense of Definition 1.1 by taking $\phi$ to be ${\bar x} = {\bar y})$. 

\begin{Remark} The collection of definable subsets of $Aut(M/A)$ is an $Aut(M/A)$-invariant Boolean subalgebra of the Boolean algebra of clopen subsets of $Aut(M/A)$.
\end{Remark} 
\begin{proof}   Clearly the definable subsets of $Aut(M/A)$ form a Boolean algebra.  
Let $X = \{\sigma: M\models\phi(\sigma({\bar b}), {\bar b})\}$ be such a definable set. Then $X$ is the union of the basic open sets $\{\sigma: \sigma({\bar b}) = {\bar c}\}$ as ${\bar c}$ ranges over realizations of $\phi({\bar x}, {\bar b})$ in $M$, so $X$ is open  (so clopen as the complement of $X$ is also definable so open).
\newline
For invariance:  Let $X\subseteq Aut(M/A)$ be defined by $\phi({\bar x}, {\bar b})$ as above. Let $\tau\in Aut(M/A)$. Then $\tau(X)$ is defined by  $\chi({\bar x}_{1}, {\bar x}_{2}, {\bar b}, \tau({\bar b}))$ where  $\chi({\bar x}_{1}, {\bar x}_{2}, {\bar y}_{1}, {\bar y}_{2})$ is $\phi({\bar x}_{1}, {\bar y}_{2})$. 

%Consider the definable subset $Aut(M/A)$: $X = \{\sigma:\models \phi(\sigma(m_{1},..,m_{n}), m_{1},..,m_{n})\}$. Then $X$ is  the union of the basic opens $\{\sigma: \sigma(m_{1},..,m_{n}) = (b_{1},..,b_{n})\}$ as $(b_{1},..b_{n})$ ranges over realizations of $\phi(x_{1},..,x_{n}, m_{1},..,m_{n})$ in $M$.  Clearly the definable subsets of $Aut(M/A)$ form a Boolean algebra.  This is enough. 

\end{proof} 

We let $Def(Aut(M/A))$ denote this $Aut(M/A)$-invariant Boolean algebra of definable subsets of $Aut(M/A)$.

\begin{Definition} (i) By a Keisler measure on $Aut(M/A)$ we mean a finitely additive probability measure on  $Def(Aut(M/A))$. 
\newline
(ii) We will say that $Aut(M/A)$ is {\em definably amenable} if there is an $Aut(M/A)$ invariant Keisler measure on $Aut(M/A)$.
\end{Definition}

There is a slight discrepancy with notation from \cite{HKP} where we used the expressions relatively definable and relatively amenable.

Recall that a topological group $G$ is said to be  amenable if for any continuous action of $G$ on a compact space $X$, $X$ supports a $G$-invariant probability measure. When $M$ is sufficiently saturated and homogeneous over $A$, the definable amenability of $Aut(M/A)$ is equivalent to the (continuous)  action of $Aut(M/A)$ on a certain type space supporting an invariant probability measure.  We will prove the same when $M$ is atomic and strongly $\omega$-homogeneous over $A$.  Hence in these cases, amenability of $Aut(M/A)$ as a topological group implies definable amenability. 

In any case, for $M$ very saturated (with respect to $|A|$), definable amenability of $Aut(M/A)$ was in \cite{HKP} one of the equivalent definitions of $T$ being {\em first order amenable} over $A$.

\section{Invariant measures}
 Recall that $M$ is atomic over $A$ if the type over $A$ of any finite tuple from $M$ is isolated. And $M$ is strongly $\omega$-homogeneous over $A$ if whenever $a,b$ are finite tuples from $M$ with the same type over $A$, then there is $\sigma\in Aut(M/A)$ such that $\sigma(a) = b$. 
If $T$ is totally transcendental (i.e. $\omega$-stable when in a countable language) then there are prime models over all sets $A$ and these are atomic and strongly $\omega$-homogeneous (over $A$).

\begin{Theorem} Suppose $M$ is atomic and strongly $\omega$--homogeneous over $A$.
Then:
\newline
(i) When $T$ is stable, $Aut(M/A)$ is uniquely definably amenable (has a  unique invariant Keisler measure). 
Moreover if $A$ is algebraically closed (in $T^{eq}$), this unique invariant Keisler measure is $\{0,1\}$-valued.
\newline
(ii) When $M = acl(A)$, $Aut(M/A)$ is also uniquely definably amenable. Moreover the definable subsets of $Aut(M/A)$ coincide with the clopen subsets of  the profinite group $Aut(M/A)$,
and the unique invariant Keisler measure on $Aut(M/A)$ coincides with the Haar measure. 
\end{Theorem}

The proof of (i) will be an adaptation  of known results for Keisler measures on automorphism groups of saturated models to the case of atomic and strongly $\omega$-homogeneous models. 
 
Let us now fix $M$ atomic and strongly $\omega$-homogeneous over $A$ and fix also an enumeration ${\bar m}$ of $M
$.  Let ${\bar x}$ be an (infinite) tuple of variables with enumeration corresponding to ${\bar m}$.  And let $p_{0}({\bar 
x}) = tp({\bar m}/A)$.  We let $S_{p_{0}}(M)$ denote the space of extensions of $p_{0}({\bar x})$ to complete types in 
variables ${\bar x}$ over $M$. So $S_{p_{0}}(M)$ is a closed subspace of the space $S_{\bar x}(M)$ of all complete types 
over $M$ in variables ${\bar x}$.  We may assume that ${\bar m}$ is enumerated by ordinals $<$ some cardinal.  Now a clopen subset of 
$S_{p_{0}}(M)$ can be written in the form $\phi({\bar x}, {\bar m})$. However only finitely many variables from ${\bar x}$ and elements of $\bar m$ actually appear in this formula.
Hence, adding if necessary dummy variables and constants we may assume the formula is of the form
$\phi(x_{\alpha_{1}}, ....,x_{\alpha_{n}}, m_{\alpha_{1}},...,m_{\alpha_{n}})$ for some $\alpha_{1} < ....< \alpha_{n}$. 

\begin{Definition}  Let $f$ be the map from the collection of clopen subsets of $S_{p_{0}}(M)$ to $Def(Aut(M/A))$ defined by:
given, as in the discussion above,  $X$  the clopen defined by $\phi(x_{\alpha_{1}},...,x_{\alpha_{n}}, m_{\alpha_{1}},...,m_{\alpha_{n}})$, let 
$f(X) = \{\sigma\in Aut(M/A): M\models \phi(\sigma( m_{\alpha_{1}},...,m_{\alpha_{n}}),  m_{\alpha_{1}},...,m_{\alpha_{n}})\}\in Def(Aut(M/A))$.
\end{Definition}

%Let $S_{\bar m, A}(M)$ denote the set of extensions of $p_{0} = tp({\bar m}/A)$ to complete types over $M$ (in the same variables).  
 %So $S_{\bar m}(M)$ is a closed subset of the space of all complete types over $M$ in the same variables. In contradistinction with conventions in the last section, let us fix a (infinite)  tuple ${\bar x}$ of variables, enumerated in the same manner as  ${\bar m}$.  (So if ${\bar m} = (m_{\alpha}: \alpha < \kappa)$ then ${\bar x} = (x_{\alpha}:\alpha<\kappa)$.) 

We first establish an analogue of Proposition 2.12(1) from \cite{HKP}, 
\begin{Lemma}  The map $f$ defined above is an isomorphism of Boolean algebras, which moreover commutes with the actions of $Aut(M/A)$.  
\end{Lemma}
\begin{proof}  We will just check the main point that if the clopen $X$ of $S_{p_{0}}(M)$ is nonempty, then $f(X)$ is nonempty.
For simplicity of notation, let us assume that $X$ is given by $\phi(x_{1},..,x_{k}, m_{1},..,m_{k})$. As $X$ is nonempty, let $p({\bar x})\in S_{p_{0}}(M)$ contain
$\phi(x_{1},..,x_{k}, m_{1},..,m_{k})$. Let ${\bar n}$ realize $p({\bar x})$ in ${\bar M}$.  So $(n_{1},..,n_{k})$ has the same type over $A$ as $(m_{1},..,m_{k})$. Let $\psi(x_{1},..,x_{k})$ isolate this type. Then the formula $\psi(x_{1},..,x_{k})\wedge\phi(x_{1},..,x_{k}, m_{1},..,m_{k})$  is realized in $M$, by $(b_{1},..,b_{k})$ say. By strong $\omega$-homogeneity of $M$ over $A$, there is $\sigma\in Aut(M/A)$ such that $\sigma(m_{1},..,m_{k}) = (b_{1},..,b_{k})$, whence $\sigma$ is in the definable subset of $Aut(M/A)$ determined by $\phi(x_{1},..,x_{k}, m_{1},..,m_{k})$ which is precisely $f(X)$.

\end{proof}

\noindent
{\bf Proof of Theorem 2.1  (i).}  By Lemma 2.3,  we only have to  find an $Aut(M/A)$-invariant finitely additive probability measure on the Boolean algebra of clopens of $S_{p_{0}}(M)$ and show it to be unique.  The latter means the same thing as finding a (unique)    $Aut(M/A)$-invariant (regular) Borel probability measure on the space $S_{p_{0}}(M)$. Furthermore this amounts to the same thing as finding a Keisler measure $\mu({\bar x})$ over $M$ which is $Aut(M/A)$-invariant and extends $p_{0}(\bar x) = tp({\bar m}/A)$ (that is every formula in $p_{0}({\bar x})$ gets $\mu$-measure $1$) and showing it is the unique such Keisler measure.  When $M$ is  the ``monster model" (i.e. sufficiently saturated and homogeneous relative to $A$) and ${\bar x}$ is a finite tuple of variables, this is precisely Remark 1.3 of \cite{invariant-measures-simple}.  Passing from the finite tuple ${\bar x}$ to an infinite tuple is unproblematic. 
But we have to say something about working with $M$ atomic and strongly $\omega$-homogeneous over $A$ in place of  the monster model $\bar M$. 
Fix a finite set $\Delta$  of $L$- formulas $\phi({\bar x}, {\bar y})$, where we may assume that ${\bar y}$ is a fixed finite tuple of variables. (Of course only finitely many variables from $
{\bar x}$ actually appear in the formulas in $\Delta$.)  Consider the set $P_{\Delta} = \{p({\bar x})|\Delta: p({\bar x})\in S_{\bar x}(M)$ is a nonforking extension of $p_{0}({\bar x})\}$.   
Fix $q\in P_{\Delta}$. Then by general stability,  $q$ is definable by formulas  $\psi_{\phi}({\bar y}, e)$ for $\phi\in \Delta$ and some $e\in acl^{eq}(A)$.  Moreover every $q'\in 
P_{\Delta}$ has defining schema  $(\psi_{\phi}({\bar y}, e'))_{\phi\in \Delta}$ for some $e'$ realizing $tp(e/A)$. By strong $\omega$-homogeneity of $M$ over $A$, for 
each such $e'$ there is $\sigma\in Aut(M/A)$ taking $e$ to $e'$ so taking $q$ to $q'$.  So $Aut(M/A)$ acts transitively on $P_{\Delta}$, and for $\chi({\bar x}, {\bar b})$ a $\Delta$-formula over $M$, define $\mu_{\Delta}(\chi({\bar x}, {\bar b}))$ to be the proportion of those $q\in P_{\Delta}$ which contain $\chi({\bar x}, {\bar b})$.  By construction, $\mu_{\Delta}$ is $Aut(M/A)$-invariant. 
Now one has to check that if $\Delta\subseteq \Delta'$ then $\mu_{\Delta'}$ agrees with $\mu_{\Delta}$ on ${\Delta}$-formulas over $M$.  This is because $P_{\Delta}$ is precisely the set of restrictions to $\Delta$ of the types in $P_{\Delta'}$.  Hence the(directed)  union of the $\mu_{\Delta}$ gives a Keisler measure $\mu({\bar x})$ over $M$ which by construction extends $p_{0}({\bar x})$ and is $Aut(M/A)$-invariant. 

In \cite{invariant-measures-simple} when $M$ was the a monster model, we obtained uniqueness of $\mu$ from  Fact 1.2 in that paper that with no restriction on the theory, any formula which forks over $A$ has measure $0$ for any global $Aut(M/A)$-invariant Keisler measure. In our current context ($M$ atomic and strongly $\omega$-homogeneous over $A$), it is likewise enough for uniqueness of the $\mu$ constructed in the previous paragraph that any formula $\phi({\bar x}, {\bar b})$ over $M$ which forks over $A$ has $\nu$-measure $0$ for any $Aut(M/A)$-invariant Keisler measure $\nu({\bar x})$ over $M$.  So suppose $\nu({\bar x})$ is a Keisler measure over $M$ and $\phi({\bar x},{\bar b})$ forks over $A$ and has positive $\nu$-measure (and we look for a contradiction).  By stability $\nu$ is {\em definable}, namely for any  $L$-formula $\chi({\bar x}, {\bar y})$ and disjoint closed subsets $C_{1}, C_{2}$ of $[0,1]$ there is a formula $\psi({\bar y},{\bar c})$ over $M$ which separates $\{{\bar b}\in M: \nu(\chi({\bar x}, {\bar b})) \in C_{1}\}$ and $\{{\bar b}\in M: \nu(\chi({\bar x}, {\bar b}))\in C_{2}\}$.   (This follows from Lemma 1.7 of \cite{Keisler} for example.) We claim that from $Aut(M/A)$-invariance of $\nu$, $\nu$ is actually definable over $A$, namely the formula $\psi({\bar y},{\bar c})$ can be chosen over $A$ in the above. Let $\theta({\bar z})$ be a formula over $A$ isolating $tp({\bar c}/A)$. Then by $Aut(M/A)$-invariance of $\nu$, and strong $\omega$-homogeneity over $A$ of $M$, we can replace $\psi({\bar y},{\bar c})$ by $\exists {\bar z}(\psi({\bar y}, {\bar z})\wedge\theta({\bar z}))$ which is over $A$. 
We can now use the same ``defining scheme" to extend $\nu({\bar x})$ to a global (over the monster model ${\bar M}$) Keisler measure $\nu'({\bar x})$. $\nu'$ will be $Aut({\bar M}/A)$-invariant (as it is definable over $A$) and $\phi({\bar x}, {\bar b})$ still forks over $A$ and has positive $\nu'$ measure.  This contradicts Fact 1.2 of \cite{invariant-measures-simple}. 

So we have proved unique definable amenability of $Aut(M/A)$. When $A$ is algebraically closed,  $p_{0} = tp({\bar m}/A)$ has a unique nonforking extension over $M$, which by above arguments is also the unique $Aut(M/A)$-invariant Keisler measure over $M$ extending $p_{0}$.  This completes the proof of Theorem 2.1(i).  \qed

\vspace{5mm}
\noindent
{\bf Proof of Theorem 2.1 (ii).}  Now suppose that $T$ is arbitrary and $M = acl(A)$. So $M$ is both atomic over $A$ and strongly $\omega$-homogeneous over $A$. In the pointwise topology, the basic opens of ${\mathcal G} = Aut(M/A)$  are of the form $\{\sigma: \sigma({\bar a}) = {\bar b}\}$ for ${\bar a}$, ${\bar b}$ finite tuples of the same length. Each such basic open is clopen and  is a coset of the subgroup $Fix({\bar a})$ which has finite index in ${\mathcal G}$ as the orbit of ${\bar a}$ under ${\mathcal G}$ is finite.  So ${\mathcal G}$ is also compact (and moreover profinite).  Any clopen in ${\mathcal G}$ is by compactness a finite union of basic opens, so  will therefore be some (finite) union of cosets of $Fix({\bar b})$ for some ${\bar b}$.  Using Remark 1.2 the clopens of $\mathcal G$ are precisely the definable subsets of $\mathcal G$. So a $\mathcal G$-invariant finitely additive measure on $Def(\mathcal G)$ is the same thing as a $\mathcal G$-invariant finitely additive measure on the clopens of $\mathcal G$, which as $\mathcal G$ is profinite  induces and is induced by a unique regular $G$-invariant Borel probability measure on $\mathcal G$, the (unique) Haar measure on the compact group $G$.

Actually in  this case one can explicitly define $\mu$, by $\mu(Fix({\bar b})) = 1/n$ where $n$ is the index of $Fix({\bar b})$ in $\mathcal G$, and check that it extends uniquely to a $\mathcal G$-invariant Keisler measure on $\mathcal G$ and that it is the only one.

\vspace{5mm}
\noindent
Recall from \cite{KPR} that first order amenability (over $A\subseteq {\bar M}$) of a complete theory $T$ (with monster model ${\bar M}$) can be defined as: any complete type $p({\bar x})$ over $A$ (where ${\bar x}$ may be an infinite tuple of variables) extends to a global $Aut({\bar M}/A)$-invariant Keisler measure.

So we obtain:
\begin{Proposition} Suppose $T$ is first order amenable over $A$ and $M$ is both atomic and strongly $\omega$-homogeneous over $A$. Then $Aut(M/A)$ is definably amenable.
\end{Proposition}
\begin{proof}  Again let $p_{0} = tp({\bar m}/A)$. The clopens of $S_{p_{0}}(M)$ are precisely the clopens of $S_{p_{0}}({\bar M})$ which are defined by formulas  over $M$. Let $\mu$ be an $Aut({\bar M}/A)$ invariant Keisler measure on $S_{\bar m}({\bar M})$ given by first order amenability of $T$ over $A$.
Then restricting $\mu$ to the clopens defined over $M$ gives the required $Aut(M/A)$-invariant Keisler measure on $Aut(M/A)$  (using Lemma 2.3). 
\end{proof}

We now want to relativise the above constructions to suitable sets $B$ with $A\subseteq B \subseteq M$, namely consider the group of $A$-elementary permutations (in the sense of the ambient model) of $B$ in place of $Aut(M/A)$.  It is convenient at this point to make the {\bf assumption} that $T$ is {\em totally transcendental} (t.t)., which is equivalent to $\omega$-stability when $T$ is countable. 
$T$ being t.t. means that in the monster model $\bar M$, every formula with parameters has ordinal valued Morley rank.  The basic fact is that when $T$ is t.t., over any set $A\subseteq \bar M$ there is a unique (up to $A$-isomorphism) prime model $M$ over $A$. $M$ is an elementary substructure of ${\bar M}$ containing $A$ and characterized, up to isomorphism over $A$ by,  $M$ is atomic over $A$, and $M$ contains no uncountable $A$-indiscernible sequence.  (And of course $M$ is strongly $\omega$-homogeneous over $A$).

\begin{Definition} Let $M$ be prime over $A$  and $A\subseteq B \subseteq M$. We say that $B$ is {\em normal over $A$} if $B$ is fixed setwise by $Aut(M/A)$, equivalently for all finite tuples ${\bar b}$, all realizations of (the isolated type) $tp({\bar b}/A)$ in $M$ are in $B$. 
\end{Definition}

The following is well-known (see Section 18.2 of \cite{Poizat-modeltheory} for example). 

\begin{Lemma} Suppose $M$ is prime over $A$.  Assume $A\subseteq B \subseteq M$ and that $B$ is normal over $A$. Then
\newline
(i) $M$ is also prime over $B$, and 
\newline
(ii) The group of permutations of $B$ induced by $Aut(M/A)$  coincides with the group of $A$-elementary permutations of $B$.  We call this group $Aut_{M}(B/A)$.
\end{Lemma}
\begin{proof}
We just mention (ii). It is enough to see that any $A$-elementary permutation $\sigma$  of $B$ extends to an $A$-automorphism $\sigma'$ of $M$.
First $\sigma$ extends to an $A$-automorphism of ${\bar M}$. Let $N = \tau(M)$. So (using (i)) both $N$ and $M$ are prime over $B$, so there is a an isomorphism $f:N\to M$ fixing $B$ pointwise. But then  $f\circ \tau$ is an automorphism of $M$ extending $\sigma$. 

\end{proof} 

We now want to point out that in the above context, the material in Sections 1 and 2 generalizes to $Aut_{M}(B/A)$.

So again we assume $M$ is prime over $A$ and $B\subseteq M$ is normal over $A$.
\begin{Definition} As  as in Definition 1.1, by a {\em definable} (in $M$) subset of $Aut_{M}(B/A)$ we mean something of the form $\{\sigma\in Aut_{M}(B/A): M\models \phi(\sigma({\bar b}), {\bar b})\}$ for some formula $\phi({\bar x}, {\bar y})\in L_{A}$ and tuple ${\bar b}$ from $B$. 
\end{Definition}

Let now ${\bar b}$ be an enumeration of $B$ and $p_{0} = tp({\bar b}/A)$. $S_{p_{0}}(M)$ is the space of complete types over $M$ extending $p_{0}$. 
The material from Sections 1 and 2 goes through with similar  proofs. 

\begin{Proposition}
(i) The collection $Def(Aut_{M}(B/A))$ of definable subsets of $Aut_{M}(B/A)$ is a Boolean algebra.
\newline
(ii)  This Boolean algebra is canonically isomorphic to the Boolean algebra of clopens of $S_{p_{0}}(M)$.
\newline
(iii) There is a unique $Aut_{M}(B/A)$-invariant finitely additive probability measure on $Def(Aut_{M}(B/A))$  (the unique Keisler measure on $Aut_{M}(B/A)$). 
\end{Proposition}

We should just say something about (ii) because on the face of it,  a clopen subset of $S_{p_{0}}(M)$ is given by a formula $\phi({\bar x}, c)$ say where $c$ is a tuple from $M$.  Without loss of generality ${\bar x} = (x_{1},..,x_{n})$ corresponds to $(b_{1},..,b_{n})$ from $B$. Let
Let $\psi(x_{1},..,x_{n})$ isolate $tp(b_{1},..,b_{n}/A)$. 
As  $T$ is  $t.t.$, $tp(c/B)$ is definable over a finite set, which without loss of generality is also $(b_{1},..,b_{n})$.  So there is a formula $\chi(x_{1},..,x_{n}, b_{1},..,b_{n})$ such that for $d_{1},..,d_{n}\in B$,   $M\models \chi(d_{1},..,d_{n}, b_{1},..,b_{n})$ iff $M\models \phi(d_{1},..,d_{n},c)\wedge \psi(d_{1},..,d_{n})$.  We may assume that $\chi(x_{1},..,x_{n}, d_{1},..,d_{n})$ implies $\psi(x_{1},..,x_{n})$.  Noting that (by normality) all realizations of $\psi$ in $M$ are in $B$, we see that the clopen of $S_{p_{0}}(M)$ given by $\phi(x_{1},..,x_{n},c)$ coincides with that given  by $\chi(x_{1},..,x_{n}, b_{1},..,b_{n})$.  Hence, as in Definition 2.2, we can define a map $f$ from the set of clopens of $S_{p_{0}}(M)$ to the set of definable subsets of $Aut_{M}(B/A)$: the clopen $\chi(x_{1},..,x_{n}, b_{1},..,b_{n})$ of $S_{p_{0}}(M)$ is taken by $f$ to $\{\sigma\in Aut_{M}(B/A): M\models \chi(\sigma(b_{1}),..,\sigma(b_{n}), b_{1},..,b_{n})\}$

\section{Observations and  compatibilities}
We will make some more observations about definability in the automorphism groups we have been considering.

\subsection{Definable subgroups}
Let $M$ be atomic and strongly $\omega$-homogeneous over $A$, and let ${\mathcal G} = Aut(M/A)$.  We here point out that the definable subgroups of ${\mathcal G}$ are precisely of the form $Fix(e)$ where $e\in M^{eq}$.

\begin{Lemma} Let $\phi({\bar x}, {\bar y})$ be an $L_{A}$ formula, and ${\bar b}$ be a tuple in $M$ such that $\phi({\bar x}, {\bar b})$ defines a subgroup $H$ of ${\mathcal G}$, i.e. 
$\{\sigma\in {\mathcal G}: M\models \phi(\sigma({\bar b}), {\bar b})\}$ is a subgroup of ${\mathcal G}$. Let $\psi({\bar x})$ isolate $tp({\bar b}/A)$. Then $\phi({\bar x}, {\bar y})$ defines an equivalence relation on the set of realizations of $\psi$  (in $M$ so also in ${\bar M}$). 
\end{Lemma}
\begin{proof} (i) Reflexivity: As the identity is in $H$, $M\models \phi({\bar b}, {\bar b})$ whereby $M\models \forall {\bar x}(\psi({\bar x}) \to \phi({\bar x},
{\bar x}))$. 
\newline
(ii) Symmetry:    Let ${\bar c}$ realize  $\psi({\bar x})\wedge \phi({\bar x}, {\bar b})$ in $M$. So (strong $\omega$-homogeneity) there is $\sigma\in {\mathcal G}$ such that $
\sigma({\bar b}) = {\bar c}$.  As $\sigma^{-1}\in {\mathcal G}$ we have $M\models \phi(\sigma^{-1}({\bar b}), {\bar b})$. Applying $\sigma$ we obtain $M\models \phi({\bar b}, {\bar 
c})$.  We have shown that $M\models \forall{\bar x}(\psi({\bar x})\wedge \phi({\bar x}, {\bar b})\to \phi({\bar b}, {\bar x}))$. As $\psi$ isolates $tp({\bar b}/A)$ we conclude
that $M\models \forall {\bar y}\forall {\bar x}(\psi({\bar y})\wedge \psi({\bar x})\wedge \phi({\bar x}, {\bar y}) \to \phi({\bar y}, {\bar x}))$.
\newline
(iii)  Transitivity: Given that we have symmetry, and the isolation of $tp({\bar b}/A)$ by $\psi$, it suffices to prove that if ${\bar c}$, ${\bar d}$ realise $\psi$ in $M$, and 
$M\models \phi({\bar c}, {\bar b})$ and $M\models \phi({\bar d},{\bar b})$ then $M\models \phi({\bar d}, {\bar c})$.  Let $\sigma, \tau\in {\mathcal G}$ be such that $\sigma({\bar b}) = {\bar c}$ and $\tau({\bar b}) = {\bar d}$.  So $\sigma$ and $\tau$ are in $H$, whereby the composition $\sigma\tau \in H$. 
Then $M\models \phi(\sigma\tau({\bar b}), {\bar b})$, namely  $M\models \phi(\sigma({\bar d}), {\bar b})$. Apply $\sigma^{-1}$ to get $M\models\phi({\bar d}, {\bar c})$ as required.
\end{proof} 

We conclude:
\begin{Proposition} The definable subgroups of $Aut(M/A)$ are precisely the subgroups $Fix(e)$ for $e\in M^{eq}$, namely $\{\sigma\in {\mathcal G}:\sigma(e) = e\}$
\end{Proposition}
\begin{proof} Let $H$ be a definable subgroup of the form $\{\sigma\in {\mathcal G}: M\models \phi(\sigma({\bar b}), {\bar b})\}$  (for $\phi({\bar x}, {\bar y})\in L_{A}$ and ${\bar b}$ a finite tuple from $M$).  Lemma 3.1 above says that $\phi$ defines an equivalence relation $E$ say on realizations of $\psi$ where $\psi$ isolates $tp({\bar b}/A)$.  Of course ${\mathcal G}$ preserves $E$ and by definition $H$ is the subgroup of ${\mathcal G}$ which fixes the $E$-class of ${\bar b}$. So take $e$ to be ${\bar b}/E$.   (On the face of it $E$ is $A$-definable, but standard methods give us a $\emptyset$-definable equivalence relation.)
The converse, that $Fix(e)$ is definable, is immediate by considering the formula defining $E$. 
\end{proof} 

\begin{Remark} (i) Consider the special case where  $T$ is $\omega$-categorical, $A = \emptyset$ and $M$ is the unique countable model of $T$. It is well-known that the subgroups $Fix(e)$ are {\em  precisely} the open subgroups of $Aut(M)$. So open subgroups and definable  subgroups coincide. On the other hand there are more clopen subsets of $Aut(M)$ than definable subsets. 
\newline
(ii) Assuming that $T$ is $t.t$, $M$ is prime over $A$, and $A\subseteq B \subseteq M$ with $B$ normal,  Proposition 3.2 extends to definable subgroups of $Aut_{M}(B/A)$. Namely definable subgroups of $Aut_{M}(B/A)$ are of the form $Fix(e)$ for some imaginary in $dcl^{eq}(B)$. 
\newline
(ii) In a future paper we will use the results above to yield  Galois correspondences in various contexts. 
\end{Remark}

\subsection{Definable Galois cohomology}
In  \cite{Pillay-Galois} we developed ``Galois cohomology" in a model-theoretic setting.  This generalized  Kolchin's constrained cohomology (Chapter VII of \cite{Kolchin-DAG})  which in turn extended classical Galois cohomology \cite{Serre}.  The general idea is to classify principal homogeneous spaces for groups (all ``defined over" some base set or field) in various categories, such as algebraic-geometric, differential algebraic-geometric, and definable. 

As in \cite{Pillay-Galois} we work in the context of $T$ an arbitrary theory, $M$ being atomic and strongly $\omega$-homogeneous over $A$.
  %The context of classiical Galois cohomology is $T = ACF$ (say characteristic $0$) and $M = acl(A)$ (the prime model over $A$). The context of

%We worked in \cite{Pillay-Galois}in  the general  context of $M$ being atomic and strongly $\omega$-homogeneous over $A$.  But here it will again be convenient to simply assume $T$ is totally transcendental (as in the last part of Section 2) and that $M_{A}$ is prime over $A$.  (This subsumes Kolchin's set-up, where the ambient theory is $DCF_{0}$ and for $A = K$ a differential field, $M_{K}$
%is also called the differential closure $K^{diff}$ of $K$, or as Kolchin says the ``constrained closure" of $K$.)

Let ${\mathcal G}$ be $Aut(M/A)$, and let $G$ be a group definable with parameters from  $A$  in the structure $M$. Note that ${\mathcal G}$ acts on $G$ (on the left say). The key notion in \cite{Pillay-Galois} was that of a {\em definable cocycle} $f$ from ${\mathcal G}$ to $G$.  The  cocycle condition simply means that $f:{\mathcal G} \to G$ satisfies $f(\sigma\cdot \tau) =  f(\sigma)\cdot\sigma(f(\tau))$, for any $\sigma, \tau$ in ${\mathcal G}$. 
The definability condition as given in \cite{Pillay-Galois} is that there is some $A$-definable function $h(-,-)$ and  some tuple ${\bar b}$ from $M$ such that for any $\sigma\in {\mathcal G}$, $f(\sigma)  = h(\sigma({\bar b}), {\bar b})$. 

As explained in \cite{Pillay-Galois} this specializes to classical Galois cohomology, taking $T = ACF$, and  $M = acl(A)$, and ``constrained cohomology" taking $T = DCF_{0}$ and $M$ the differential closure of $A$ (prime model over $A$). 

The main point to make here is that the notion of a definable map from ${\mathcal G}$ to $G$ can  be seen to be a special case of the definitions in Section 1, suitably extended.

\begin{Definition} (With earlier notation, and assumptions.)
Let $X$ be a set definable with parameters from $A$, in the structure $M$. Let ${\bar m}$ be an enumeration of $M$ and let ${\bar x}, {\bar y}$ be tuples of variables with corresponding enumerations. 
By a definable subset of $Aut(M/A) \times X$, we mean $\{(\sigma, c)\in Aut(M/A)\times X: M\models \phi(\sigma({\bar m}), {\bar m},  c)\}$ for some formula $\phi({\bar x}, {\bar y}, z)$ over $A$. 
\end{Definition}  

\begin{Lemma} (${\mathcal G} = Aut(M/A)$.)  Let again $X$ be a set definable in $M$ over $A$. Let $f:{\mathcal G}\to X$ be a function. Then the following are equivalent:
\newline
(i) the graph of $f$ is a definable subset of ${\mathcal G}\times X$ in the sense of Definition 3.4.
\newline
(ii) There is finite tuple ${\bar b}$ from $M_{A}$ and an $A$-definable (partial) function $h(-,-)$ such that for any $\sigma\in {\mathcal G}$, $f(\sigma) = h(\sigma({\bar b}), {\bar b})$. 
\end{Lemma}
\begin{proof} (ii) implies (i).  We may assume that ${\bar b} = (m_{1},..,m_{n})$, and $h$ has variables $x_{1},..,x_{n}, y_{1},..,y_{n}$.
Let $\phi({\bar x}, {\bar y}, z)$ be the formula $z = h({\bar x}, {\bar y})$. Then $\phi$ witnesses that the graph of $f$ is a definable subset of ${\mathcal G}\times X$ (as in Definition 3.4). 
\newline
(i) implies (ii).  So suppose $\phi({\bar x}, {\bar y},z)$ witnesses that the graph of $f$ is definable in the sense of Definition 3.4. We may assume that $m_{1},..,m_{n}$ begin the enumeration of ${\bar m}$ and that $x_{1},..,x_{n},y_{1},...,y_{n}$ are the ${\bar x}$, ${\bar y}$ variables in $\phi$. 
So the graph of $f$, $X$ say,  is $\{(\sigma,c)\in {\mathcal G}\times X: M\models \phi(\sigma(m_{1},..,m_{n}), m_{1},..,m_{n},c)\}$. 

Let $\psi(x_{1},..,x_{n})$ be a formula isolating $tp(m_{1},..,m_{n}/A)$. 
\newline
{\em Claim.} $M \models   \forall {\bar x}(\psi({\bar x}) \to \exists^{=1}z\phi({\bar x}, m_{1},..,m_{n}, z))$
\newline
{\em Proof of claim.}  Let $(m_{1}',...,m_{n}')$ realize $\psi(x_{1},..,x_{n})$ in $M_{A}$. Then there is $\sigma\in {\mathcal G}$ such that $\sigma(m_{1},..,m_{n}) = (m_{1}',..,m_{n}')$.  Then $f(\sigma)\in X$ has to be the unique  $c\in X$ such that  $M\models \phi(m_{1}',...,m_{n}', m_{1},..,m_{n},c)$.  \qed

\vspace{2mm}
\noindent
As $\psi(y_{1},..,y_{n})$ isolates $tp(m_{1},..,m_{n}/A)$ it  follows that 
$M\models \forall {\bar x}\forall{\bar y}(\psi({\bar x})\wedge \psi({\bar y}) \to \exists ^{=1}z\phi({\bar x}, {\bar y}, z))$.
Define $h({\bar x},{\bar y}) = z$ iff $\models\phi({\bar x}, {\bar y}, z) \wedge \psi({\bar x})\wedge \psi({\bar x})$. 
\end{proof}

Lemma 3.5 shows that the two notions of definable map from ${\mathcal G}$ to $X$ coincide. 

\vspace{2mm}
\noindent
Let us just note in passing that two definable cocycles $f_{1}:{\mathcal G} \to G$ and $f_{2}:{\mathcal G}\to G$ are said to  be {\em cohomologous} if for some $b\in G$, for all $\sigma\in {\mathcal G}$, $f_{2}(\sigma) = b^{-1}\cdot f_{1}(\sigma)\cdot\sigma$. $H^{1}_{def}({\mathcal G}, G)$ is the set of equivalence classes (modulo the equivalence relation of being cohomologous).  The main point is that $H^{1}_{def}({\mathcal G}, G)$ classifies the collection of definable right torsors for $G$ defined over $A$, up to $A$-definable isomorphism (commuting with the action of $G$). 

\vspace{5mm}
\noindent
After making the additional assumptions that $T$ is $t.t$ and $M$ is prime over $A$, the above relativises to replacing $M$ by  some normal $B$ wth $A\subseteq B \subseteq M$. So first a subset of $Aut_{M}(B/A)\times X(B)$ is called definable if it is of the form 
$\{(\sigma, c): M\models \phi(\sigma({\bar b}), {\bar b}, c)\}$ for some $\phi$ over $A$. And when $f:Aut_{M}(B/A) \to X(B)$ is a function who graph is definable, we call $f$ a definable function. Again when $X$ is a group $G$ we have definable cocycles $Aut_{M}(B/A) \to G(B)$ and definable cocycles up to being cohomologous. In fact this is the level of generality of Chapter VII, Section 2 of Kolchin's \cite{Kolchin-DAG} in the context of the theory $DCF_{0}$. Putting this part of Kolchin's work into the general model-theoretic setting will appear in a forthcoming paper of David Meretzky.

\vspace{2mm}
\noindent
We also have a natural notion of a definable subset of the Cartesian power ${\mathcal G}^{n}$, namely $\{(\sigma_{1},...,\sigma_{n}): M\models \phi(\sigma_{1}({\bar b}),...,\sigma_{n}({\bar b}), {\bar b}\})$ for some formula $\phi$ over $A$ and finite tuple ${\bar b}$ from $M$.   In a forthcoming paper with Krzysztof Krupinski we describe when the (graph of) the group operation on ${\mathcal G}$ is definable.

\subsection{Differential Galois theory}
We will focus on the Galois theory of linear differential equations, often called the Picard-Vessiot theory, although we could work with more general families of algebraic differential equations and differential field extensions (such as the strongly normal theory or the generalized strongly normal theory) or even work in a fairly abstract model theoretic environment such as in Section 2 of \cite{Leon-Sanchez-Pillay-definablegaloistheory}.

The ambient first order theory is here $DCF_{0}$ the theory of differentially closed fields of characteristic $0$, in the language of rings together with a symbol $\partial$ for the derivation. $DCF_{0}$ is complete with quantifier elimination and elimination of maginaries. It is also $\omega$-stable. 

The Galois theory of linear differential equations has origins in the 19th century but was put on a firm footing (and generalized) by Kolchin in the 20th century. See \cite{Kolchin-book} and also \cite{vdP-Singer}.  There are plenty of references for the model-theoretic point of view on this theory, but let us just mention \cite{Kamensky-Pillay} for more background.

We fix a ``monster model" ${\mathcal U}$. $K, L,...$ etc. denote ``small" differential subfields.  The prime model over $K$ is also known as the differential closure $K^{diff}$ of $K$.  All of differential Galois theory over $K$ takes place inside $K^{diff}$.  We assume that the field $C_{K}$ of constants of $K$ is algebraically closed, from which it follows that $C_{K} = C_{K^{diff}}$. 

We fix $K$. By a homogeneous linear differential equation over $K$ we mean something of the form 
\newline
(*) $\partial Y = AY$,
\newline
 where $Y$ is an $n\times 1$ column vector of indeterminates and $A$ is now an $n\times n$ matrix with coefficients from $K$.  We will fix such an equation, and we let $L$ be the differential field generated by $K$ and the coordinates of all solutions of (*) in $K^{diff}$.  $L$ is called the Picard-Vessiot (or PV) extension of $K$  corresponding to the equation (*).
$L$ is visibly normal, i.e. invariant under $Aut(K^{diff}/K)$.  And by quantifier elimination $Aut_{\partial}(L/K)$ (the group of automorphisms of the differential field $L$ which fix $K$ pointwise) coincides with what we called above
$Aut_{K^{diff}}(L/K)$.  And $C_{K} = C_{L}$. 

Let ${\mathcal V}$ denote the solution set of (*) in $K^{diff}$ (a collection of $1\times n$ vectors). Then ${\mathcal V}$ is an $n$=dimensional  $C_{K}$-vector space. Hence $L$ is generated, as a differential field over $K$ by any basis of ${\mathcal V}$. Let us fix such a basis ${\bar b} = (b_{1},..,b_{n})$, a nonsingular $n\times n$-matrix over $L$. 
For each $\sigma\in Aut_{\partial}(L/K)$, $\sigma({\bar b})$ is also a $C_{K}$-basis of ${\mathcal V}$, hence there is unique $c_{\sigma}\in GL_{n}(C_{K})$ sich that $\sigma({\bar b}) = {\bar b}c_{\sigma}$  (matrix multiplication in $GL_{n}(L))$. 

\begin{Fact} The map $h$ taking $\sigma$ to $c_{\sigma}$ yields an isomorphism between $Aut_{\partial}(L/K)$ and an algebraic (i.e. definable in the algebraically closed field $(C_{K},+,\times)$) subgroup $G$ of $GL_{n}(C_{K})$. 
\end{Fact}

\begin{Proposition} (i) $h$ induces an isomorphism between the Boolean algebra of definable subsets of $Aut_{\partial}(L/K)$, and the Boolean algebra of definable, in $(C_{K},+,\times)$, subsets of $G$.
\newline
(ii)  Under $h$, the unique $Aut_{\partial}(L/K)$-invariant Keisler measure on $Aut_{\partial}(L/K)$ corresponds to the unique translation invariant Keisler measure on $G$ over $C_{K}$
\end{Proposition}
\begin{proof}
(i) As  ${\bar b} = \{b_{1},..,b_{n}\}$ generates $L$ over $K$, any definable subset $X$  of $Aut(L/K)$ can be written as $\{\sigma: \models \phi(\sigma({\bar b}), {\bar b})$ for some formula $\phi$ over $K$.  Then $h(X) = \{g\in G: \models \phi({\bar b}\cdot g, {\bar b})$ which is a definable in $K^{diff}$ subset of $G$, which by stability is a definable (in $(C_{K},+, \times)$) subset of $G$. Any definable (in $(C_{K},+,\cdot)$) subset of $G$ arises in this way.
For let $Y$ be such.  For each $g\in Y$, there is unique $\sigma\in Aut_{\partial}(L/K)$ such that $\sigma({\bar b}) = {\bar b}g$, and the set of all such $\sigma$ as $g$ varies in $Y$ is
$\{\sigma\in Aut(L/K): K^{diff}\models \exists y\in Y(\sigma({\bar b}) = {\bar b}y)\}$ which is a definable subset of $Aut_{\partial}(L/K)$.   One  can check that $h$ induces not only a bijection between the two Boolean algebras but also a Boolean algebra isomorphism.
\newline
(ii) One has to first show that for a definable subset $X$ of $Aut_{\partial}(L/K)$ and $\sigma\in Aut_{\partial}(L/K)$, 
$h(\sigma(X)) = h(\sigma)h(X)$, and then it is immediate  that $h$ takes  any $Aut_{\partial}(L/K)$ invariant Keisler measure on $Aut_{\partial}(L/K)$ to a translation invariant Keisler measure on $G$ (as a definable group in a model of $ACF_{0}$).  Note that the latter is just the average of the Dirac measures on the cosets of $G^{0}$ in $G$. 
\end{proof}

\vspace{5mm}
\noindent
Finally we extend the above observations to $K^{PV}$, the union of all the $PV$ extensions of $K$.  Otherwise said $K^{PV}$ is the differential field generated by all solutions in $K^{diff}$ of all homogeneous linear differential equations over $K$. Clearly $K^{PV}$ is also normal (over $K$ in $K^{diff}$).  $Aut_{\partial}(K^{PV}/K)$ is called the absolute $PV$-group of $K$, in analogy with $Aut(K^{alg}/K)$, the absolute Galois group of $K$.  $Aut_{\partial}(K^{PV}/K)$ has the structure of a proalgebraic group over $C_{K}$. See \cite{vdP-Singer}. But we give a few details here.

Note that if $L_{1},..,L_{m}$ are $PV$ extensions of $K$, then so is their (differential) field compositum. Also if $L_{1}$ and $L_{2}$ are $PV$ extensions of $K$ with $L_{1}\subseteq L_{2}$, and if $G_{1}$, $G_{2}$ are the corresponding linear algebraic groups over $C_{K}$ given by Fact 3.6, then the inclusion of $L_{1}$ in $L_{2}$ gives a canonical surjective homomorphism (of algebraic groups) from $G_{2}$ to $G_{1}$. 
Hence we have a directed system of linear algebraic groups over $C_{K}$, $G_{i}$ for $i\in I$ and  $\pi_{i,j}: G_{i} \to G_{j}$ for $i\geq j$. The inverse limit of this system is $\{(a_{i})_{i\in I}: a_{i}\in G_{i}(C_{K})$ for all $i$, and $\pi_{i,j}(a_{i}) = a_{j}$ whenever $i\geq j$\}. This can also be described as the group of $C_{K}$-rational points of an affine group scheme over $C_{K}$.   Let us call this group $G$.  We have canonical surjective homomorphisms $\pi_{i}:G\to G_{i}$ for all $i\in I$. 

\begin{Fact} The isomorphisms $h$ between automorphism groups (over $K$) of $PV$ extensions of $K$ and linear algebraic groups over $C_{K}$ cohere and yield   an isomorphism  $h_{PV}$ between $Aut_{\partial}(K^{PV}/K)$ and $G$
\end{Fact}

\begin{Definition} By a definable subset of $G$ we mean the preimage under some $\pi_{i}$ of a definable (in $(C_{K},+,\times)$) subset of $G_{i}$.
\end{Definition}

\begin{Proposition} (i) $h_{PV}$ induces a bijection which we also call $h_{PV}$, between definable subsets of $Aut_{\partial}(K^{PV}/K)$ (in the sense of Definition 2.7) and definable subsets of $G$ (in the sense of Definition 3.9).
\newline
(ii) For $\sigma\in Aut_{\partial}(K^{PV}/K)$, and $X$ a definable subset of  $Aut_{\partial}(K^{PV}/K)$, 
$h_{PV}(\sigma(X)) = h_{PV}(\sigma)(h_{PV}(X))$.
\newline
(iii) $h_{PV}$ takes the unique $Aut_{\partial}(K^{PV}/K)$-invariant Keisler measure on $Aut_{\partial}(K^{PV}/K)$ to the unique $G$-invariant Keisler measure on the pro- (or $*$-)definable (in $ACF_{0}$) group $G$
\end{Proposition}
\begin{proof} Left to the reader.
\end{proof}

\end{document}